\documentclass[11pt,reqno]{amsart}
\usepackage{amssymb,latexsym}
\usepackage{graphicx}
\usepackage{fancyhdr}
\usepackage[all]{xy}
\usepackage{enumerate}
\theoremstyle{plain}
\numberwithin{equation}{section}
\newtheorem{thm}{Theorem}[section]
\newtheorem{theorem}[thm]{Theorem}

\newtheorem{corollary}[thm]{Corollary}

\newtheorem{lemma}[thm]{Lemma}
\newtheorem{remark}[thm]{Remark}

\newtheorem{proposition}[thm]{Proposition}
\newtheorem*{theorem*}{Theorem}

\begin{document}

\setcounter{page}{1}

\title[Classification of finite dimensional nilpotent Lie superalgebras by their multiplier]{Classification of finite dimensional nilpotent Lie superalgebras by their multiplier}
\author[Nayak]{Saudamini Nayak}
\address{Institute of Mathematics \& Applications\\
         Andharua, 
          Bhubaneswar-751029\\
                India}
\email{ anumama.nayak07@gmail.com}


\begin{abstract}\label{abstract}
Let $L$ be a nilpotent Lie superalgebra of dimension $(m\mid n)$  and
$s(L) = \frac{1}{2}[(m + n - 1)(m + n -2)]+ n+ 1 - \dim \mathcal{M}(L)$, 
where $\mathcal{M}(L)$ denotes the Schur multiplier of $L$.  Here $s(L)\geq 0$ and the structure of all non-abelian nilpotent Lie superalgebras with $s(L)=0$ is known((\cite{Nayak2019}). This paper is devoted to obtain all  nilpotent Lie superalgebras when $s(L) \leq 2$. Further, we apply those results to list all non-abelian nilpotent Lie superalgebras $L$ with $ t(L) \leq 4$.   
   
\end{abstract}

\maketitle

\section{Introduction}\label{intro}

For a given group $G$, the notion of the Schur multiplier $\mathcal{M}(G)$ arose from the work of I. Schur on projective representation of groups as the second 
cohomology group with coefficients in $\mathbb{C}^{*}$ \cite{Kar1987}. Let $0\longrightarrow R \longrightarrow F \longrightarrow G \longrightarrow 0$ be a free presentation of the group $G$, 
then it can be shown by Hopf's formula that $\mathcal{M}(G) \cong R\cap [F,F]/[R,F]$. Green in \cite{Green1956} proved that for every finite $p$-groups $G$ of order $p^{n}$, the order of Schur multiplier of $G$ is equal to $p^{\frac{1}{2}n(n-1)-t(G)}$ for some non-negative integer $t(G)$. Classification of $p$-groups in terms of the order of their Schur multipliers have been studied by several authors, for instance the structure of $G$ are obtained when $t(G) \in \{1, 2, 3, 4, 5\}$ in \cite{Berkovich1991, Ellis1999, Gaschutz1967, Nir2009, Zhou1994}.

With a similar motivation, as of Green, Moneyhun in \cite{Moneyhun1994}, showed that for  nilpotent Lie algebras $A$ of dimension $n$, $\dim \mathcal{M}(A) = \frac{1}{2}n(n-1)-t(A)$, where $t(A)$ is some non-negative integer. Further, Batten et. al., in \cite{Batten1993, BS1996, BMS1996}
obtained structure of all such $A$ with $t(L) \leq 2$. In \cite{Hardy1998} and later in \cite{Hardy2005} with a different method to \cite{BMS1996}, the cases $t(A)= 3, 4, 5, 6$ and $t(A)=7, 8$ were considered, respectively. Moreover classification of filiform nilpotent Lie algebras has been done for $t(A) \leq 16$ in \cite{ Bosko2010, GJK1998}.

Niroomand and Russo in \cite{NR2011}, have improved the upper bound of  Moneyhun. Precisely, they showed that for non-abelian nilpotent Lie algebras $A,\,\dim \mathcal{M}(A)\leq \frac{1}{2}(n-1)(n-2)+1-s(A)$ for some $s(A) \geq 0$. Then the structure of  Lie algebras $A$ are obtained when $s(A) \in\{ 0, 1, 2, 3\}$ in  \cite{NR2011,Nir2011, Saeedi2017}.

The theory of Lie superalgebras and Lie supergroups has many applications in various areas of Mathematics and Physics. In 1975, Kac in \cite{KAC1977} offered a comprehensive description of the mathematical theory of Lie superalgebras, and has established the classification of all finite-dimensional simple Lie superalgebras over an algebraically closed field of characteristic zero. However the classification of finite dimensional nilpotent Lie superalgebras is still an open problem as it is the case for Lie algebras. Till today nilpotent Lie superalgebras $L$ where $\dim L \leq 5$ over real and complex fields are known \cite{alv,back, ah, na}. All Lie superalgebras up to dimension four over real field are known since 1978 \cite{back}, and in  this list there are no simple Lie superalgebras. Then nilpotent Lie superalgebras up to dimension five over both real and complx fields are classified by Hegazi \cite{ah}.  Again classification of all five dimensional Lie supealgebras are done over complex field \cite{na}. Recently Isabbel et. al., have given classification of nilpotent Lie superalgebras $L$ with $\dim L \leq 5$ in \cite{alv}, which includes some missing five dimensional superalgebras in \cite{ah, na}. Furthermore classification of some low dimensional filiform Lie superalgebras are done in \cite{gilg2001, gomez2004}.
 
\smallskip

Recently the theory of Schur multiplier is extended to Lie superalgebras \cite{Nayak2018,Nayak2019,NRK2019}. In \cite{Nayak2019} the author found an upper bound for the dimension of Schur multiplier of Lie superalgebras. Precisely, it has been shown  \cite{Nayak2019} that for  non-abelian nilpotent Lie superalgebras $L$ of dimension $(m\mid n)$ and $\dim L'=(r\mid s)$, $\dim \mathcal{M}(L)\leq \frac{1}{2}[(m+n+r+s-2)(n+m-r-s-1)]+n+1$. Now consider $\dim \mathcal{M}(L)=\frac{1}{2}[(m+n+r+s-2)(n+m-r-s-1)]+n+1-s(L).$ 
  But as the bound for the dimension of multiplier of $L$ is decreasing with respect to $r, s$, so we define $s(L)$ as,
  \begin{equation}\label{eq2}
 s(L) = \frac{1}{2}(m+n-2)(m+n-1)+n+1- \dim \mathcal{M}(L),
   \end{equation}
   and $s(L) \geq 0$.
  In this paper, our aim is to obtain all nilpotent Lie superalgebras with $s(L)\leq 2$ and for this we use the classification resuts in \cite{alv}. 
  
\subsection{Main results}

 We need the following notation and terminology to state the main results. Here we denote an abelian Lie superalgebra of dimension $(m \mid n)$ by $A(m\mid n)$, the Heisenberg Lie superalgebra of even center with dimension $(2m+1 \mid n)$ by $H(m, n), \; m+n \geq 1$ and the Heisenberg Lie superalgebra of odd center with dimension $(n \mid n+1)$ by $H_n, \; n\geq1$. 
 
 \begin{theorem}\label{th4.9a}
 Let $L$ be a nilpotent Lie superalgebra with $\dim L =(m \mid n)$. Then $s(L) =0$ if and only if 
\[L \cong  H(1 \mid 0) \oplus A(m-3 \mid n),\; A(2 \mid 0), \; A(0 \mid 2), \; A(1 \mid 1).\]
 \end{theorem}
 
\begin{theorem}\label{th4.9}
Let $L$ be a nilpotent Lie superalgebra with $\dim L =(m \mid n)$. Then $s(L) =1$ if and only if 
\[L \cong  L_{5, 0}^2,\; A(1\mid 0), \; A(0 \mid 1).\]
\end{theorem}

\begin{theorem}\label{th4.10}
Let $L$ be a nilpotent Lie superalgebra with $\dim L =(m \mid n)$. Then $s(L) = 2$ if and only if 
\begin{align*}
L \cong &H(p, q)\oplus A(m-2p-1 \mid n-q), \;( p+q \geq 2), H(0, 1)\oplus A(m-1 \mid n-1),\\
&  H_1\oplus A(m - 1 \mid n-2), L_{4, 0}^1,  L_{5, 0}^2 \oplus A(1\mid 0), L_{5, 0}^2 \oplus A(0\mid 1).
\end{align*}
\end{theorem}

\begin{theorem}\label{th410}
Let $L$ be a non-abelian nilpotent Lie superalgebra, $\dim L=(m \mid n)$. Then
\begin{enumerate}[(i)]
\item $t(L)=1$ if and only if $L \cong H(1, 0)$.
\item $t(L)=2$ if and only if $L \cong H(1, 0)\oplus A(1 \mid 0)$ or $H(0,1)$.
\item $t(L)=3$ if and only if $L \cong H(1, 0)\oplus A(2 \mid 0)$, $H(0,1)\oplus A(1 \mid 0)$, $H(0,2)$, $H_{1}$ or $H(0,1)\oplus A(0 \mid 1)$.
\item $t(L)=4$ if and only if $L \cong H(1, 0)\oplus A(3 \mid 0)$, $L_{5,0}^2$, $L_{4,0}^1$, $H(0,3)$, $H_{1} \oplus A(1 \mid 0)$, $H(0,1)\oplus A(1 \mid 1)$, $H(0,1)\oplus A(0 \mid 2)$, $H(0, 1)\oplus A(2 \mid 0),\, H_{1}\oplus A(0 \mid 1)$  or $H_{1} \oplus A(2 \mid 0)$.
\end{enumerate}
\end{theorem}

\smallskip

The organization of the paper is as follows, after the brief introduction, we provide some useful notation in Section \ref{sec 1}. Section \ref{sec 2} contains some auxiliary results about multiplier of Lie superalgebras. Here, we also explicitly compute multiplier of  $H_{n}, \; n \geq 1$. In Section \ref{sec 3}, we characterize nilpotent Lie superalgebras of dimension less than or equal to five with some specific conditions. In Section \ref{sec 4} we classify all finite dimensional nilpotent Lie superalgebras $L$ with $s(L) \leq 2$ and also with $ t(L) \leq 4$.  
 
\section{Preliminaries}\label{sec 1}

Let $\mathbb{Z}_{2}=\{\bar{0}, \bar{1}\}$ be a field. A $\mathbb{Z}_{2}$-graded vector space $V$ is simply a direct sum of vector spaces $V_{\bar{0}}$ and $V_{\bar{1}}$, i.e., $V = V_{\bar{0}} \oplus V_{\bar{1}}$. It is also referred as a superspace. We consider all vector superspaces and superalgebras are over $\mathbb{C}$. Elements in $V_{\bar{0}}$ (resp. $V_{\bar{1}}$) are called even (resp. odd) elements. Non-zero elements of $V_{\bar{0}} \cup V_{\bar{1}}$ are called homogeneous elements. For a homogeneous element $v \in V_{\sigma}$, with $\sigma \in \mathbb{Z}_{2}$ we set $|v| = \sigma$ is the degree of $v$. If $\dim V_{\bar{0}} = m$ and $\dim V_{\bar{1}} = n$, we say that  dimension of $V$ is $(m\mid n)$. A vector subspace $U$ of $V$ is called $\mathbb{Z}_2$-graded subsuperspace(or subspace) if $U= (V_{\bar{0}} \cap U) \oplus (V_{\bar{1}} \cap U)$. We adopt the convention that whenever the degree function appeared in a formula, the corresponding elements are supposed to be homogeneous.

A {\it Lie superalgebra} (see \cite{KAC1977, Musson2012}) is a superspace $L = L_{\bar{0}} \oplus L_{\bar{1}}$ with a bilinear mapping
$ [., .] : L \times L \rightarrow L$ satisfying the following identities:

\begin{enumerate}
\item $[L_{\alpha}, L_{\beta}] \subset L_{\alpha+\beta}$, for $\alpha, \beta \in \mathbb{Z}_{2}$ ($\mathbb{Z}_{2}$-grading),
\item $[x, y] = -(-1)^{|x||y|} [y, x]$ (graded skew-symmetry),
\item $(-1)^{|x||z|} [x,[y, z]] + (-1)^{ |y| |x|} [y, [z, x]] + (-1)^{|z| |y|} [z,[ x, y]] = 0$ (graded Jacobi identity),
\end{enumerate}
for all $x, y, z \in L$. 

\smallskip
It then follows that $L_{\bar{0}}$ is a Lie algebra and $L_{\bar{1}}$ is a $L_{\bar{0}}$-module. If $L_{\bar{1}} = 0$, then $L$ is  a Lie algebra, though in general  Lie superalgebras are not Lie algebras. A  superspace $L$ with identically vanishing Lie brackets is trivially a Lie superalgebra, called abelian Lie superalgebra. Also Lie superalgebras without even parts($L_{\bar{0}} = 0$), are  abelian, as $[x, y] = 0$ for all $x, y \in L_{\bar{1}}$. A subsuperalgebra (subalgebra) of $L$ is a $\mathbb{Z}_{2}$-graded subspace which is closed under bracket operation. A $\mathbb{Z}_{2}$-graded subspace $I$ is a {\it graded ideal} of $L$ if $[I,L]\subseteq I$. For instance 
\[Z(L) = \{z\in L : [z, x] = 0\;\mbox{for all}\;x\in L\}\] 
is a graded ideal and it is called the {\it center} of $L$. 
Let $Z_0(L) = <0>$, then the $i$th center of $L$ is defined inductively by 
\[\frac{Z_i(L)}{Z_{i-1}(L)} = Z\left(\frac{L}{Z_{i-1}(L)}\right)\]
for all $i\geq 1$. 
Clearly, $Z_1(L) = Z(L)$.  If $I$ and $J$ are graded ideals of $L$, then so is $[I, J]$. 

\smallskip

 By a {\it homomorphism} between superspaces $f: V \rightarrow W $ of degree $|f|\in \mathbb{Z}_{2}$, we mean a linear map satisfying $f(V_{\alpha})\subseteq W_{\alpha+|f|}$ for $\alpha \in \mathbb{Z}_{2}$. In particular, if $|f| = \bar{0}$, then the homomorphism $f$ is called homogeneous linear map of even degree. A Lie superalgebra homomorphism $f: L \rightarrow M$ is a complex homogeneous linear map of even degree such that $f[x,y] = [f(x), f(y)]$ holds for all $x, y \in L$. For an ideal $I$ of $L$, the quotient Lie superalgebra $L/I$ inherits a canonical Lie superalgebra structure such that the natural projection map becomes a homomorphism. 
 
\smallskip
 
 The notion of {\it epimorphisms, isomorphisms} and {\it auotomorphisms} have the obvious meaning. For Lie superalgebra $L$, if $\dim L_{\bar{0}} = m$ and $\dim L_{\bar{1}} = n$, we say that  dimension of $L$ is $(m\mid n)$. But simply writing $\dim L= t$,  we mean, the basis of  $L$ has t homogeneous elements.  
 
\smallskip

For Lie superalgebra $L = L_{\bar{0}} \oplus L_{\bar{1}}$ the descending central sequence is defined by $L^{0} = L, L^{k+1} = [L^{k}, L]$ for all $k\geq 0$. If $L^{k+1} = \{0\}$ for some $k$, the Lie superalgebra is called nilpotent and $k$ is called the nilpotency class if $L^{k} \neq 0$. The subalgebra $L^{1}=[L,L]=:L'$ is called derived subalgebra of $L$, and $L$ is abelian if and only if $L'=0$. Abelian Lie superalgebras are nilpotent. Heisenberg Lie superalgebras are example of non-abelian nilpotent Lie sueralgebras.

\smallskip

Heisenberg Lie superalgebras $L$ are  nilpotent Lie superalgebra with $1$-dimensional homogeneous center such that  $L' \subseteq Z(L)$  \cite{RSS2011}. If $L'=Z(L)$ and $\dim L'=1$ then $L$ is called special Heisenberg.  All finite dimensional special Heisenberg Lie superalgebras are of two types namely, $H(m, n)$ and $H_n$, $0 \leq m \in \mathbb{N}$ and $1 \leq  n \in \mathbb{N}$, where,

$H(m, n) = H_{\bar{0}}\oplus H_{\bar{1}}$
\begin{align*}
H_{\bar{0}}&=\{ x_{1},\ldots,x_{m},x_{m+1},\ldots,x_{2m},\, z\mid [x_{i},x_{m+i}]=z, i=1,\ldots,m\}\\
H_{\bar{1}}&=\{y_{1},\ldots,y_{n}\mid [y_{j},y_{j}]=z, j=1,\ldots,n\},
\end{align*}
and
  \begin{equation*}
H_{n} =<x_{1},\ldots,x_{n}>\oplus< y_{1}, \ldots, y_{n}, z>\quad \mbox{with}\quad [x_{i},y_{i}]=z, i=1,\ldots,n. 
\end{equation*}
Here $H(m, n)$ is called special Heisenberg Lie superalgebra with one dimensional even center and $H_{n}$ is called with one dimensional odd center.

An extension of a Lie superalgebra $L$ is a short exact sequence 
\begin{equation}\label{eq1'}
\xymatrix{
0 \ar[r] & M\ar[r]^{e} & K\ar[r]^{f} & L \ar[r] & 0}. 
\end{equation}

If $L$ is a Lie superalgebra generated by a $\mathbb{Z}_{2}$-graded set $X = X_{\bar{0}} \cup X_{\bar{1}}$  and $\phi : X \rightarrow L$ is an even degree map,
then there exists a free Lie superalgebra $F$ and  $\psi: F \rightarrow L$ extending $\phi$. Let $R = \mbox{Ker} (\psi)$. The extension 
\begin{equation}\label{eq0}
0 \longrightarrow R \longrightarrow F \longrightarrow L \longrightarrow 0
\end{equation} 
is called a {\it free presentation} of $L$ and is denoted by $(F, \psi)$ \cite{Musson2012}. With this free presentation of $L$,  the {\it Schur  multiplier} of $L$ denoted by  $\mathcal{M}(L)$, and is defined as
$$ 
\mathcal{M}(L) = \frac{[F,F]\cap R}{[F, R]}.
$$
A {\it central extension} of $L$ is an extension \eqref{eq1'} such that $M \subseteq Z(K)$. The central extension is said to be a {\it stem extension} of $L$ if $M \subseteq Z(K)\cap K' $. Finally, we call the stem extension a {\it stem cover} if $M \cong \mathcal{M}(L)$ and in this case $K$ is said to be a cover of Lie superalgebra $L$. 

\smallskip

Moreover the stem extension \ref {eq1'} is {\it maximal} if every epimorphism of any other stem extension of $L$ onto $0 \longrightarrow M \longrightarrow  K \longrightarrow   L \longrightarrow 0$ is necessarily an isomorphism.
The author in \cite{Nayak2018} has proved that stem covers do exist for each Lie superalgebras and that the maximal stem extensions of finite dimensional Lie superalgebras $L$ (or, equivalently dimension of $K$ is maximal) are precisely stem covers of $L$. Hence the stratergy that we use to find dimension of multiplier of Lie superalgebras is to find maximal stem extensions of those Lie superalgebras.

\section{Auxiliary results}\label{sec 2}

In this section, at first we list some known results which we use throughout.

\begin{theorem}[See \cite{Nayak2019}]\label{th3}
Let $L=L_{\bar{0}} \oplus L_{\bar{1}}$ be a Lie superalgebra with 
$\dim \left( L/Z(L)\right) = (m \mid n)$. Then 
$$\dim L'\leq \frac{1}{2}\left[(m+n)^2 + (n-m)\right].$$
\end{theorem}

\begin{theorem}[See \cite{Nayak2019}]\label{th3.1}
Let $L$ be a Lie superalgebra with 
$\dim L = (m \mid n)$. Then 
\[\dim \mathcal{M}(L) \leq \frac{1}{2}\left[(m+n)^2 + (n-m)\right].\]
 Further, equality holds if and only if $L$ is abelian.
\end{theorem}
As a consequence 
\begin{equation}\label{eq1}
\dim \mathcal{M}(L)= \frac{1}{2}[(m+n)^{2}+(n-m)]-t(L),
\end{equation} for some non-negative integer $t(L)$ and $t(L)=0$ if and only if $L$ is abelian.

\begin{theorem}[See \cite{Nayak2019}]\label{th3.2}
Let $L$ be a finite dimensional Lie superalgebra with $\dim L=(m \mid n)$, $K \subseteq Z(L)$ be an graded ideal and $H = L/K$. Then
\begin{enumerate}
\item $\dim \mathcal{M}(L) + \dim (L'\cap K) \leq \dim \mathcal{M}(H) + \dim \mathcal{M}(K) + \dim (H/H'\otimes K/K');$
\item $ \dim \mathcal{M}(L) + \dim (L'\cap K)  \leq \frac{1}{2}[(m+n)^2 +(n-m)].$
\end{enumerate}
\end{theorem}

\begin{theorem}[See \cite{Nayak2019}]\label{th3.4}
Let $A$ and $B$ be two finite dimensional Lie superalgebras. Then 
\begin{equation*}
\dim \mathcal{M}(A\oplus B) = \dim \mathcal{M}(A) + \dim \mathcal{M}(B) + \dim (A/A'\otimes B/B'). 
\end{equation*}
\end{theorem}
\begin{theorem}[See \cite{Nayak2019}]\label{th3.5}
Let $H(m,n)$ be a special Heisenberg Lie superalgebra with even center of dimension $(2m+1\mid n)$, then 
$$
\dim \mathcal{M}(H(m,n))=\begin{cases}
2m^2-m-1+2mn+n(n+1)/2 & \mbox{if}\ m + n \geq 2 \\
0 & \mbox{if }\ m=0, n=1\\

2  &\mbox{if }\ m=1, n=0.
\end{cases}
$$
\end{theorem}

The following result is an immediate consequence of Theorem \ref{th3.1}, Theorem \ref{th3.4} along with Theorem \ref{th3.5}.
\begin{corollary}\label{cor3.6}
Suppose $K$ is an abelian Lie superalgebra of dimension $(m-2p-1\mid n-q)$ and suppose $L:= H(p,q) \oplus K$ is a Lie superalgebra. Then 
$$
\dim \mathcal{M}(L)=\begin{cases}
\frac{1}{2}\left[(m+n)^2-(3m+n)\right] & \mbox{if}\ p+q \geq 2\;\; \mbox{or}\;\;p=0, q=1 \\
 2 + \frac{1}{2}\left[(m+n)^2-(3m+n)\right] & \mbox{if }\ p = 1, q = 0.
\end{cases} 
$$
\end{corollary}

Now, we explicitly compute the multiplier of special Heisenberg Lie superalgebra with odd center.
\begin{theorem}\label{th3.7}
Let $H_{n}$, $n \geq 1$ be a special Heisenberg Lie superalgebra with odd center and $\dim H_{n}=(n\mid n+1)$, then 
$$
\dim \mathcal{M}(H_{n})=\begin{cases}
2n^{2}-1 & \mbox{if}\ n \geq 2 \\
 2 & \mbox{if }\ n= 1.
\end{cases} 
$$
\end{theorem}

\begin{proof}
 First suppose that $n>1$ and we want to find a stem cover for the special Heisenberg Lie superalgebra $H_{n}$. Let $W$ be a superspace generated by

$$\{w_{i}, \eta_{i}, \eta_{i}', \gamma_{i,j}, \delta_{i,j}, \beta_{i,j}\}.$$ Consider a superspace $C$ with a basis, 

$$\{x_{1}, x_{2}, \cdots , x_{n}, z, y_{1}, y_{2}, \cdots, y_{n}\}$$
 where $x_i,\, (i =1, \ldots, n)$ are even elements and  $z, y_j,\, (j =1, \ldots, n)$ are odd elements and put $K= C+ W$. We want to define a Lie superbracket in $K$ which turns $K$ into a Lie superalgebra. Define
\begin{align*}
[x_{i}, y_{i}] & = z + w_{i},\;\; 1\leq i \leq n,\\
[x_{i}, z] & = \eta_{i},\;\; 1\leq i \leq n,\\
[y_{i},z] &= \eta'_{i}, \;\; 1\leq i \leq n,\\
[x_{i}, x_{j}] & = \gamma_{i, j},\; \mathrm{for}\; 1 \leq i \leq j \leq n,\\
[y_{i}, y_{j}] & = \delta_{i,j},\; \mathrm{for}\; 1 \leq i \leq j \leq n, \\
[x_{i}, y_{j}] &= \beta_{i, j}, \; \mathrm{for}\; 1 \leq i \neq j \leq n.
\end{align*}
The above multiplications remain unaffected even if we replace $z$ by $\zeta$, put $w_{1} = 0$ and replace $w_{i}$ by $\hat{w}_{i} = w_{i} - w_{1}$ for $i \geq 2$. So $\zeta=[x_{1}, y_{1}] = z + w_{1}$ and
 for $1 < i \leq n$ we have $[x_{i},y_{i}]= z+ \hat{w}_{i}$. Again, for $1 \leq i \leq n$
\begin{align*}
 [x_{i}, \zeta] &= \eta_{i}\\
  [y_{i}, \zeta]  &= \eta'_{i}.
   \end{align*}

Using the graded Jacobi identity we have
\begin{align*}
0 & =J(x_{1}, y_{1}, x_{i})\\
  & =[[x_{1},y_{1}], x_{i}]+[[y_{1}, x_{i}],x_{1}]+[[x_{i}, x_{1}],y_{1}]\\
  & = [\zeta , x_{i}] - [\beta_{i, 1}, x_{1}] - [\gamma_{1, i}, y_{1}]  \\
  & = [\zeta, x_{i}] + 0 + 0 = -\eta_{i}, 
\end{align*}
 for $2 \leq i \leq n$ and by choosing $J(x_{2}, y_{2}, x_{1})=0$ we get $\eta_{1}=0$.
Similarly, we can make $\eta_{i}'$ equal to zero by considering $J(x_{1}, y_{1}, y_{i})=0$ when $1\leq i \leq n$.
Therefore $K$ is a Lie superalgebra and  
\begin{align*}
w_{i} &\quad 1< i\leq n,\\
\gamma_{i, j} & \quad 1\leq i < j \leq n,\\
\delta_{i,j} & \quad 1 \leq i \leq j \leq n, \\
\beta_{i, j} & \quad 1 \leq i \neq j \leq n, 
 \end{align*}
 forms a basis of $W$. Here observe that $\gamma_{i,j}, \delta_{i,j}\in W_{\bar{0}}$ and $w_{i}, \beta_{i,j} \in W_{\bar{1}}$. As $K$ is of maximal dimension, so $0\longrightarrow W \longrightarrow K \longrightarrow H_{n} \longrightarrow 0$ 
is a stem cover for $H_{n}$. Thus
 
\begin{align*}
 \dim W = \dim \mathcal{M}(H_{n}) &= n-1+\binom{n}{2}+\binom{n}{2}+n+n^{2}-n \\
  & = 2n^{2}-1.
 \end{align*}
 Precisely, $\dim W=(n^{2} \mid n^{2}-1)$. Now, suppose that $n=1$. Let $W$ and $C$ be the vector superspaces with bases $\{w, \eta_{1},\eta_{2}, \eta_{3}, \eta_{4}\}$ and $\{x, z, y\}$( $x$ is an even element and $z, y$ are odd elements) respectively. We show, $K=W+C$ is a Lie superalgebra.
 Define the the super brackets as
$$
[x,y]=z+w \;, [y,y] = \eta_{1}, [z,z]=\eta_{2} \;, [y,z]=\eta_{3} \;, [x, z]=\eta_{4},
$$
 which makes K a Lie superalgebra. If we take $z+w=\zeta$ the rest brackets remain unaffected. Then taking $J(x,y,y)=0$ and $J(x,y,\zeta)=0$ we get $\eta_{3}=0$ and $\eta_{2}=0$ respectively. The basis of $W$ is $\{ \eta_{1}, \eta_{4}\}$ where $\eta_{1}$ is even element and $\eta_{4}$ is odd element. Therefore, $0\longrightarrow W \longrightarrow K \longrightarrow H_{1} \longrightarrow 0$ is a stem cover
 for $H_{1}$ and thus 
 $$
 \dim W = \dim \mathcal{M}(H_{1})=2.
 $$
\end{proof}
The following corollary is an immediate consequence of Theorem \ref{th3.1}, Theorem \ref{th3.4} along with Theorem \ref{th3.7}.
\begin{corollary}\label{cor3.8}
Suppose $K$ is an abelian Lie superalgebra of dimension $(m-p\mid n-p-1)$ and suppose $L:= H_{p} \oplus K$ is a Lie superalgebra. Then 
$$
\dim \mathcal{M}(L)=\begin{cases}
 -1+\frac{1}{2}\left[(m+n)^2-(3m+n)\right] & \mbox{if}\ p \geq 2 \\
  \frac{1}{2}\left[(m+n)^2-(3m+n)\right] & \mbox{if }\ p= 1.
\end{cases} 
$$
\end{corollary}

If $L = L_{\bar{0}} \oplus L_{\bar{1}}$ is a nilpotent Lie superalgebra of dimension $(m\mid n)$ and $\dim L'= (r\mid s)$ with $r \geq 1$, then it is known that
\begin{equation}\label{eq3'}
\dim \mathcal{M}(L)\leq \frac{1}{2}\left[(m + n + r + s - 2)(m + n - r -s -1) \right] + n + 1.
\end{equation}

 Here replace the condition $r \geq 1$ with $r+s \geq 1$. By proceeding as in the proof \cite[Theorem 5.1]{Nayak2019} and using Corollary \ref{cor3.8} it is easy to see that Eq (\ref{eq3'}) still holds. Specifically,
 
\begin{theorem}\label{th3.9}
Let $L = L_{\bar{0}} \oplus L_{\bar{1}}$ be a nilpotent Lie superalgebra of dimension $(m\mid n)$ and $\dim L'= (r\mid s)$ with $r+s \geq 1$. Then 
\begin{equation}
\dim \mathcal{M}(L)\leq \frac{1}{2}\left[(m + n + r + s - 2)(m + n - r -s -1) \right] + n + 1.
\end{equation}
Moreover, if $r+s = 1$, then the equality holds if and only if 
$$
L \cong 
H(1, 0) \oplus A(m-3 \mid n)
$$
where $A(m-3 \mid n)$ is an abelian Lie superalgebra of dimension $(m-3 \mid n)$,  and $H(1, 0)$ is special Heisenberg Lie superalgebra of dimension $(3 \mid 0)$.
\end{theorem}

Define the function $s(L)$ as 
 \begin{equation}\label{eq3}
   s(L)= \frac{1}{2}(m+n-2)(m+n-1)+n+1-\dim \mathcal{M}(L).
   \end{equation}
  
 From \eqref{eq1} and \eqref{eq3} we have $$t(L)= m+n-2+s(L).$$ It seems classifying nilpotent Lie superalgebras $L$ by $s(L)$ in turn leads to classification of $L$ in terms of $t(L)$. 

\section{Characterization of nilpotent Lie superalgebras}\label{sec 3}
The classification of nilpotent Lie superalgebras of dimension $\leq 5$ have been studied by several authors \cite{ alv,back, ah, na}. Among them, \cite{alv} corrects the classification in \cite{ah, na}. Further, we note that the classification in \cite{back} was studied over real field. 

In this section we give some characterization of non-abelian nilpotent Lie superalgebras $L$.

\begin{theorem}\label{th4.1}
Let $L=L_{0} \oplus L_{1}$ be a non-abelian nilpotent Lie superalgebra with $\dim L=(m \mid n), \dim Z(L) = (a \mid b)$ and $\dim L'=(r \mid s)$. \\
 \mbox{Case I}. If $m+n=2$, then $L \cong H(0,1)$.\\
 \mbox{Case II}.  $m + n = 3$.
 \begin{enumerate}
 \item If $a+ b =2$, then there is no Lie superalgebra with $r+s =2$.
 \item If $a+ b =2$ and $r + s = 1$, then $L \cong H(0,1) \oplus A(1 \mid 0)$ or $H(0,1) \oplus A(0 \mid 1)$.
 \item If $a+ b = 1$, then there is no Lie superalgebra with $r +s =2$.
 \item   If $a+ b =1$ and $r + s = 1$, then $L \cong H(1,0)$ or $H(0, 2)$ or $H_1$.
 \end{enumerate}
 Moreover, \begin{align*}
 & \dim \mathcal{M}(H(0,1)) = 0, \dim \mathcal{M}(H(0,2)) = 2, \dim \mathcal{M}(H(1,0)) = 2,\\
 & \dim \mathcal{M}(H(0,1) \oplus A(1\mid 0)) = 1, \dim \mathcal{M}(H(0,1) \oplus A(0\mid 1)) = 2, \dim \mathcal{M}(H_1) = 2.
 \end{align*}
 \end{theorem}
 
 \begin{proof}
Using the classification of nilpotent Lie superalgebras $L$ in \cite{alv, ah}, along with the assumptions on $L$, it  is isomorphic to one of the Lie superalgebras in the statement. Further, by applying Theorems \ref{th3.1}, \ref{th3.4}, \ref{th3.5} and \ref{th3.7}, the dimensions of multipliers of 
\[H(1, 0), H(0,1), H(0, 2), \\H_1, H(0, 1)\oplus A(1\mid 0), H(0, 1)\oplus A(0\mid 1)\] are as mentioned.
 \end{proof}

 \begin{table}[htp]\label{tab1}
\caption{}
\begin{center}
\begin{tabular}{|c | c | c|}
\hline \hline 
Notation& Basis & Relation\\ \hline 
$L_{2, 2}^{1}$ & $<x_{0}, y_{0}>\oplus <x_{1}, y_{1}>$ & $[y_{0}, y_{1}]=x_{1},[y_{1}, y_{1}] = x_{0}$ \\ \hline 
$L_{2, 2}^{2}$ & $<x_{0}, y_{0}>\oplus <x_{1}, y_{1}>$ & $[x_1, x_1] = x_0, [y_1, y_1] = y_0$\\ \hline 
$L_{2, 2}^{3}$ & $<x_{0}, y_{0}>\oplus <x_{1}, y_{1}>$ & $[x_1, y_1] = x_0, [y_1, y_1] = y_0$ \\ \hline \hline
\end{tabular}
\end{center}
\label{}
\end{table}%

 \begin{theorem}\label{th4.2}
Let $L$ be a non-abelian nilpotent Lie superalgebra with $\dim L=(m \mid n), \dim Z(L) = (a \mid b)$ and $\dim L' = (r\mid s)$ such that $m+ n = 4, a+b =2$. 
 \begin{enumerate}
 \item If $r +s =2$ and $L' = Z(L)$ then $L\cong L_{2,2}^{1}$ or $L_{2,2}^{2}$  or $L_{2,2}^{3}$.
 \item If $r + s = 1$, then $L \cong H(1,0) \oplus A(1 \mid 0)$ or $H(1,0) \oplus A(0 \mid 1)$ or $H(0,2) \oplus A(1 \mid 0)$ or $H_1 \oplus A(1\mid 0)$ or $H_1 \oplus A(0\mid 1)$.
 \end{enumerate}
 Moreover, \begin{align*}
 &\dim \mathcal{M}(L_{2, 2}^{1}) = 2, \dim \mathcal{M}(L_{2,2}^{2}) =  2, \dim \mathcal{M}(L_{2,2}^{3}) =  2,
  \dim \mathcal{M}(H(1\mid 0) \oplus A(1\mid 0)) =  4, \\
  &\dim \mathcal{M}(H(1\mid 0) \oplus A(0\mid 1)) =  5, 
 \dim \mathcal{M}(H(0, 2)\oplus A(1\mid 0)) = 4,  \dim \mathcal{M}(H_1 \oplus A(1\mid 0)) = 4.
 \end{align*}
\end{theorem}

 \begin{proof}
 Again using the classification results of nilpotent Lie superalgebras $L$ in \cite{alv, ah}, $L$ is isomorphic to one of the Lie superalgebras as listed. 
 
 \smallskip
 
 Now we compute the dimension of multiplier of $L_{2, 2}^{1} := <x_0, y_0> \oplus<x_1, y_1> $ where $[y_0, y_1] = x_1$ and $[y_1, y_1] = x_0$. Let $W$ be a superspace generated by $\{ \eta_1, \cdots, \eta_8\}.$ Consider a superspace $C$ with a basis, $\{x_{0}, y_{0}, x_{1}, y_{1}\}$
 where $x_0,y_0$ are even elements and  $x_1,y_1$ are odd elements and put $K= C+ W$. Define the Lie superbracket in $K$ as
\begin{align*}
 [x_{0},y_{0}]=\eta_1, &[x_{0},x_{1}] =\eta_{2}, [x_{0},y_{1}] =\eta_{3}, [y_{0},x_{1}] =\eta_{4},[y_{0},y_{1}] = x_{1}+ \eta_{5}, \\
&[x_{1},x_{1}] =\eta_{6}, 
[y_{1},y_{1}] = x_{0}+\eta_{7} , [x_{1}, y_{1}] =\eta_{8},
\end{align*}
 which turns $K$ into a lie superalgebra. 
  A change of variable allows that $\eta_5 =\eta_7 =0$ and this does not affect the rest brackets. Using the graded Jacobi identity on all possible triples  of basis element of $C$ we find $ \{\eta_1, \eta_4\}$ is a basis for $W$.
 Hence
\begin{equation}\label{eq1'}
\xymatrix{
0 \ar[r] & W\ar[r] & K+C \ar[r] & L_{2,2}^{1} \ar[r] & 0}. 
\end{equation}
is maximal stem extension of $L$,
   and $\dim \mathcal{M}(L_{2,2}^1)=2$. Similarly, we have computed the dimensions of multipliers of rest the Lie superalgebras. 
\end{proof}

 \begin{table}[htp]\label{tab2}
\caption{}
\begin{center}
\begin{tabular}{|c| c| c|}
\hline \hline 
Notation & Basis & Relation\\ \hline
$L_{4,0}^{1}$ & $<x_{0}, y_{0}, z_0, v_0>$ & $[x_{0}, y_{0}]=z_{0}, [y_{0}, z_{0}]=v_{0}$\\ \hline 
$L_{1, 3}^{1}$ & $<x_{0}>\oplus< x_{1},y_{1}, z_{1}>$ & $[x_{0}, x_{1}]=y_{1}, [x_{0}, y_{1}]=z_{1}$\\ \hline 
\end{tabular}
\end{center}
\label{}
\end{table}%

\begin{theorem}\label{th4.3}
Let $L$ be a non-abelian nilpotent Lie superalgebra of $\dim L=(m \mid n), \dim Z(L) = (a \mid b)$ and $\dim L' = (r\mid s)$ such that $m+ n = 4, a+ b =1$. 
 \begin{enumerate}
 \item If $r+s=2$ then $L \cong L_{4, 0}^{1}$ or $L_{1,3}^{1}$. 
 \item If $r+s=1$ then $L \cong H(1, 1)$ or $H(0, 3)$. 
 \end{enumerate}
 Moreover, \begin{align*}
 &\dim \mathcal{M}(L_{4, 0}^{1}) = 2, \dim \mathcal{M}(H(1,1)) = 3,  \\
 & \dim \mathcal{M}(H(0,3)) = 5, \dim \mathcal{M}(L_{1, 3}^{1}) = 3.
 \end{align*}
 \end{theorem}
 
 \begin{table}[htp]
\caption{}
\label{tab33}
\begin{center}
\begin{tabular}{|c| c| c|}
\hline \hline 
Notation & Basis & Relation\\ \hline
$L_{5,0}^{1}$ & $<x_{0}, y_{0}, z_0, v_0, w_0>$ & $[x_{0}, y_{0}]=z_{0}, [x_{0}, z_{0}]=v_{0}, [y_{0}, w_{0}]=v_{0}$\\ \hline 
$L_{4, 1}^{1}$ & $<x_{0}, y_{0}, z_0, v_0>\oplus<x_1>$ & $[v_{0}, z_{0}]=y_{0}, [x_{1}, x_{1}]=x_{0}, [y_{0}, v_{0}]=x_{0}$\\ \hline 
$L_{4, 1}^{2}$ & $<x_{0}, y_{0}, z_0, v_0>\oplus<x_1>$ & $[x_0, y_{0}] = z_0, [ v_{0}, y_{0}] = x_0, [x_{1}, x_{1}]=z_{0}$\\ \hline 
$L_{2, 3}^{1}$ & $<x_{0}, y_{0}>\oplus<x_1, y_1, z_1>$ & $[x_0, y_{1}] = x_1, [ x_{0}, z_{1}] = y_1, [y_{0}, z_{1}]=x_{1}$\\ \hline \hline
\end{tabular}
\end{center}
\label{}
\end{table}%

\begin{theorem}\label{th4.5}
Let $L$ be a non-abelian nilpotent Lie superalgebra with $\dim L=(m \mid n)$, $ \dim Z(L) = (a \mid b)$ and $\dim L' = (r\mid s)$ such that $ m+n=5, a+b=1$ and $r+s=2$. Then $L$ is isomorphic to one of the nilpotent Lie superalgberas listed in Table \ref{tab33}. 

Moreover, $\dim \mathcal{M}(L_{5,0}^{1})=4, \dim \mathcal{M}(L_{4,1}^{1})=4, \dim \mathcal{M}(L_{4,1}^{2})=4, \dim \mathcal{M}(L_{2,3}^{1})= 4$.
\end{theorem}

 \begin{table}[htp]
\caption{}
\label{tab4}
\begin{center}
\begin{tabular}{|c| c| c|}
\hline \hline 
Notation & Basis & Relation\\ \hline
$L_{5,0}^{2}$ & $<x_{0}, y_{0}, z_0, v_0, w_0>$ & $[x_0, y_0] = v_0$ and $[x_0, z_0] = w_0$\\ \hline 
$L_{4, 1}^{3}$ & $<x_{0}, y_{0}, z_0, v_0>\oplus<x_1>$ & $[x_{0}, y_{0}]=z_{0}, [x_{1}, x_{1}]=v_{0}$\\ \hline 
$L_{3, 2}^{1}$ & $<x_{0}, y_{0}, z_0>\oplus<x_1, y_1>$ & $[y_{0}, z_{0}]=x_{0}, [y_{0}, y_{1}]=x_{1}$\\ \hline 
$L_{2, 3}^{4}$ & $<x_{0}, y_{0}>\oplus<x_1, y_1, z_1>$ & $[x_1, x_{1}] = x_0, [ y_{1}, y_{1}] = y_0, [z_{1}, z_{1}]=x_{0}$\\ \hline 
$L_{2, 3}^{5}$ & $<x_{0}, y_{0}>\oplus<x_1, y_1, z_1>$ & $[x_1, x_{1}] = x_0, [ y_{1}, y_{1}] = y_0, [z_{1}, z_{1}]=x_{0} + y_0$\\ \hline 
$L_{2, 3}^{6}$ & $<x_{0}, y_{0}>\oplus<x_1, y_1, z_1>$ & $[x_1, y_{1}] = x_0, [ y_{1}, y_{1}] = 2y_0, [y_{1}, z_{1}]=y_{0}$\\ \hline 
$L_{2, 3}^{7}$ & $<x_{0}, y_{0}>\oplus<x_1, y_1, z_1>$ & $[x_1, y_{1}] = x_0, [ y_{1}, y_{1}] = 2y_0, [z_{1}, z_{1}]=x_{0}$\\ \hline 
$L_{2, 3}^{8}$ & $<x_{0}, y_{0}>\oplus<x_1, y_1, z_1>$ & $[x_1, y_{1}] = x_0, [ y_{1}, y_{1}] = 2y_0, [z_{1}, z_{1}]=x_{0}+y_0$\\ \hline 
$L_{2, 3}^{9}$ & $<x_{0}, y_{0}>\oplus<x_1, y_1, z_1>$ & $[x_1, y_{1}] = x_0, [ y_{1}, y_{1}] = 2y_0, [y_{1}, z_{1}]=y_0, [z_1, z_1] = x_0$\\ \hline 
$L_{2, 3}^{10}$ & $<x_{0}, y_{0}>\oplus<x_1, y_1, z_1>$ & $[x_0, z_{1}] = x_1, [ y_{1}, y_{1}] = y_0$\\ \hline 
$L_{2, 3}^{11}$ & $<x_{0}, y_{0}>\oplus<x_1, y_1, z_1>$ & $[x_0, z_{1}] = x_1, [ y_{1}, z_{1}] = y_0$\\ \hline
$L_{2, 3}^{12}$ & $<x_{0}, y_{0}>\oplus<x_1, y_1, z_1>$ & $[x_0, z_{1}] = x_1, [ y_{1}, y_{1}] = y_0, [z_1, z_1] =y_0$\\ \hline
$L_{2, 3}^{13}$ & $<x_{0}, y_{0}>\oplus<x_1, y_1, z_1>$ & $[x_0, z_{1}] = x_1, [ y_{0}, z_{1}] = y_1$\\ \hline 
$L_{1, 4}^{1}$ & $<x_{0}>\oplus<x_1, y_1, z_1, v_1>$ & $[x_0, y_1] = x_1$ and $[x_0, v_1] = z_1$\\ \hline \hline
\end{tabular}
\end{center}
\label{}
\end{table}%

\begin{theorem}\label{th4.6}
Let $L$ be a non-abelian nilpotent Lie superalgebra with $\dim L=(m \mid n), m+n=5$, $\dim Z(L) = (a \mid b)$ and $\dim L' = (r\mid s)$ such that $r+s=2$ and $L' = Z(L)$. Then $L$ is isomorphic to one of the nilpotent Lie superalgberas listed in Table \ref{tab4}.
Moreover, 
\begin{align*}
&\dim \mathcal{M}(L_{5,0}^2)= 6, \dim \mathcal{M}(L_{4,1}^{3})=5, \dim \mathcal{M}(L_{3,2}^{1})=6,  
 \dim \mathcal{M}(L_{2,3}^{4})= 4, \\
 &\dim \mathcal{M}(L_{2,3}^{5})=4, \dim \mathcal{M}(L_{2,3}^{6})=4, \dim \mathcal{M}(L_{2,3}^{7})=4, \dim \mathcal{M}(L_{2,3}^{8})=4, \\
 &\dim \mathcal{M}(L_{2,3}^{9})=5, \dim \mathcal{M}(L_{2,3}^{10})= 4,
\dim \mathcal{M}(L_{2,3}^{11})=5, \dim \mathcal{M}(L_{2,3}^{12})= 4, \\
&\dim \mathcal{M}(L_{2,3}^{13})= 4, \dim \mathcal{M}(L_{1,4}^{1})= 6.
\end{align*}
\end{theorem}

\section{Classification of nilpotent Lie superalgebras with respect to $s(L) \leq 2$} \label{sec 4}

Suppose $L$ is a finite dimensional abelian Lie superalgebra. Below we show that there is no non-zero finite dimensional abelian Lie superalgebras $L$ with $s(L) = 2$. 

\begin{lemma}\label{lem4.1}
There is no $(m \mid n)$-dimensional abelian Lie superalgebra with $m+n >1$, and $s(L) =1 $ or $s(L)=2$.
\end{lemma}
\begin{proof}
Suppose $L$ is abelian Lie superalgebra with $\dim L=(m \mid n)$, $m+n >1 $.
Using \eqref{eq3} and Theorem \ref{th3.1}, we have
\begin{align*}
\frac{1}{2}[(m+n-2)(m+n-1)]+n+1-s(L) &= \frac{1}{2}[(m+n)^{2}+n-m]
\end{align*} 
which implies $s(L) = 2-(m+n)$. Hence the result follows.
\end{proof}
 
 \begin{corollary}\label{cor l =2}
 Let $L$ be an abelian Lie superalgbera. Then $s(L)=0$ if and only if $\dim L = 2$.
 \end{corollary}
 
  \begin{corollary}\label{cor l=1}
 Let $L$ be an abelian Lie superalgbera. Then $s(L)=1$ if and only if $\dim L = 1$.
 \end{corollary}
Now consider $L$ is a finite dimensional nilpotent Lie superalgebra such that its derived subalgebra is one dimensional. The complete structure of such Lie superalgebras are given below.
 
\begin{proposition}\label{prop4.4}
 Let $L=L_{\bar{0}}\oplus L_{\bar{1}}$ be a nilpotent Lie superalgebra of dimension $(m \mid n)$. If $\dim L' = (1 \mid 0)$, then for some non-negative integers $p, q $ with $p+q \geq 1$,
 \begin{equation*}
  L \cong H(p , q)\oplus A(m-2p-1 \mid n-q).
 \end{equation*}
Moreover, if $p=0, q=1$ or $p+q \geq 2$ then $s(L)=2$.
If $\dim(L')=(0 \mid 1)$ then for some $p \geq 1$

 \begin{equation*}
  L \cong H_p \oplus A(m-p \mid n-p-1).
 \end{equation*}
Moreover, if $p \geq 2$ then $s(L)=3$ and for $p=1$ we get $s(L)=2$.
\end{proposition}
\begin{proof}
 Suppose $L$ is nilpotent with $\dim L'=(1 \mid 0)$, then $L' \subseteq Z(L)$. Let $H/L'$ be a complement to $Z(L)/L'$ in $L/L'$. Then $L= H+Z(L)$ and hence $L'=H'$. Clearly
 $H'\subseteq Z(H)$ and further $H\cap Z(L)=Z(H) \subseteq H'$. We get $Z(H)=H'=L'$. Let $A$ be complement to $L'$ in $Z(L)$. Thus $L' \oplus A=Z(L)$, which implies $L \cong H \oplus A$. Hence $H$ satisfies $H'=Z(H)$ and $\dim (H')=(1 \mid 0)$. That is $H$ is a special Heisenberg Lie superalgebra  with even center and $A$ is abelian Lie superalgebra. For some non-negative integers $p, q$ with $p+q \geq 1$, let
  $\dim  H=(2p+1 \mid q)$, then $H \cong H(p , q)$, and therefore $\dim (A)=(m-2p-1 \mid n-q)$, as required. For $p+q \geq 2$, or $p=0, q=1$, using Theorems \ref{th3.1}, \ref{th3.4} and \ref{th3.5} we have
 \begin{align}\label{eq4}
  &\dim \mathcal{M}(L)\\ \nonumber
  &= \dim  \mathcal{M}(H(p, q))+ \dim \mathcal{M}(A(m-2p-1|n-q))+\dim (H(p, q)/H'(p, q) \otimes A)\\ \nonumber
  &= \frac{1}{2}[(m+n)^{2}-(3m+n)].
 \end{align}
Now putting \eqref{eq3} in \eqref{eq4}, we get $s(L)=2$. Now consider $L$ with $\dim L'= (0 \mid 1)$. Continuing as above we get
$L \cong H \oplus A$. Here $H$  satisfies $H'=Z(H)$ and $\dim (H')=(0 \mid 1)$. So $H$ is a special Heisenberg Lie superalgebra with odd center and $A$ is abelian. For some $p\geq 1$ with $\dim H=(p\mid p+1)$, 
 $H \cong H_p$, and hence $\dim (A)=(m-p \mid n-p-1)$ as required. When $p \geq 2$ (respectively $p=1$), using Corollary \ref{cor3.8} and Eq. \eqref{eq3}, we have $s(L)=3$ (respectively $s(L)=2)$.
\end{proof}

Assume $L$ is a finite dimensional nilpotent Lie superalgebra and $\dim L'>1$. But if we take $L$ with $s(L)=1$ or $2$, then this further put restrictions on possible values of $\dim L'$.

\begin{proposition}\label{prop4.2}
 Let $L$ be a non-abelian nilpotent Lie superalgebra with $\dim L'=(r\mid s)$. Then there is no nilpotent Lie superalgebra with $r+s \geq 3$, and $s(L)=1$ or $s(L)=2$.
 \end{proposition}
 \begin{proof}
  Suppose $L$ is $(m \mid n)$ dimensional Lie superalgebra with $\dim L'=(r\mid s)$. We have
  \begin{equation}\label{eq4.1}
   \dim \mathcal{M}(L)= \frac{1}{2}(m+n-2)(m+n-1)+n+1-s(L),
  \end{equation}
  and again
   \begin{equation}\label{eq4.2}
    \dim \mathcal{M}(L)  \leq \frac{1}{2}[(m+n+1)(m+n-4)]+n+1.
   \end{equation}
   From \eqref{eq4.1} and \eqref{eq4.2}, we deduce $ 1-s(L)  \leq -2$,
which doesn't hold if $s(L)=1$ or $s(L)=2$.
 \end{proof}

In the next two results it is shown that for nilpotent Lie superalgebras $L$ whose $\dim L' = 2$, dimension of $Z(L)$ is $\leq 3$ (resp. $\leq 4)$ with $s(L)=1$ (resp. $s(L)=2)$.
\begin{proposition}\label{prop4.3}
Let $L$ be a nilpotent Lie superalgebra with $\dim Z(L) = (a \mid b), \dim (L')=(r\mid s), r+s=2$, and $s(L)=1$.
  \begin{enumerate}
 \item There is no nilpotent Lie superalgebra $L$  when $a+b \geq 3$.
 \item There is no nilpotent Lie superalgebra $L$  when $a+b=2$ with $L'\neq Z(L)$. 
 \end{enumerate}
\end{proposition}
\begin{proof}
Suppose $L$ is a nilpotent Lie superalgebra with $\dim L'=(r\mid s)$, $r+s=2$. Also $\dim  Z(L)=(a \mid b)$, $a+b \geq 3$, and $s(L)=1$. Choose a central graded ideal $I$ with $\dim I=(1\mid 0)$ or $(0|1)$ of $L$ such that $I \cap L'=0$. 
Now, $\dim (L/I)'=(p\mid q)$ where $p+q=2$. By
 taking in to account all possibilities of $(r\mid s)$ and $(p\mid q)$ and considering $\dim I=(1|0)$ we have
 \begin{equation}\label{eq4.3}
  \dim \mathcal{M}(L/I) \leq \frac{1}{2}[(m+n-1)(m+n-4)]+n+1. \\
 \end{equation}
 If $\dim I=(0 \mid 1)$ then
 \begin{equation}\label{eq4.4}
  \dim \mathcal{M}(L/I) \leq \frac{1}{2}[(m+n-1)(m+n-4)]+n. \\
 \end{equation}
Further using Theorem \ref{th3.2}, 
\begin{equation}\label{eq4.5}
 \dim \mathcal{M}(L) \leq \dim \mathcal{M}(L/I)+(m+n-3).\\
\end{equation}
Now when $\dim I=(1\mid 0)$, we obtain from \eqref{eq4.3}, and \eqref{eq4.5} 
\begin{align*}
 \frac{1}{2}[(m+n-2)(m+n-1)]+n+1-s(L) &\leq \frac{1}{2}[(m+n-1)(m+n-4)]+n+1+m+n-3
\end{align*}
which implies $ 2  \leq s(L)$ and this is a contradiction to the assumption $s(L)=1$. Similarly taking $\dim I=(0 \mid1)$ we get $2\leq s(L)$, again a contradiction. If $\dim  Z(L)=2$ and $L'\neq Z(L)$ then also it is easy to see no such superalgebra exists.
\end{proof}

\begin{proposition}\label{prop4.4a}
Let $L$ be a nilpotent Lie superalgebra with $\dim Z(L) = (a \mid b), \dim L'=(r\mid s), r+s=2$ and
$s(L)=2$. 
\begin{enumerate}
\item There is no nilpotent Lie superalgebra $L$ when $a+b \geq 4$.
\item There is no nilpotent Lie superalgebra $L$ when $a+b=3$ and $L' \not\subseteq Z(L)$. 
\item There is no nilpotent Lie superalgebra $L$ when $a+b=2$ and $L'\neq Z(L)$. 
  \end{enumerate}
\end{proposition}
\begin{proof}
 We only proof the second part. 
Suppose $\dim Z(L)=(a\mid b), \ a+b \geq 3$ and $L'\nsubseteq Z(L)$. Choose a central ideal $I$ with $\dim I= (c \mid d), \;c+d=2$ such that $I \cap L'=0$. Let $\dim I=(2 \mid 0)$. So we have $\dim( L/I)=(m-2 \mid n)$ and $\dim (L/I)^{'}=(2\mid 0)$. Then,
\begin{align*}
\dim \mathcal{M}(L/I) &\leq \frac{1}{2}(m+n-2)(m+n-5)+n+1 \\
&=\frac{1}{2}[(m+n)^{2}-7(m+n)]+n+6.
\end{align*}
Further,
\begin{align*}
\dim \mathcal{M}(L) &\leq \dim \mathcal{M}(L/I)+1+2(m+n-4)\\
 &\leq \frac{1}{2}[(m+n)^{2}-7(m+n)]+3n+2m.
\end{align*}
Putting for $\dim \mathcal{M}(L)=\frac{1}{2}[(m+n)^{2}-3(m+n)]+n+2-s(L)$ in the above inequality we get $s(L) \geq 3$. By considering $\dim I=(1 \mid 1)$ or $(0 \mid 2)$ we again get $s(L) \geq 3$.
\end{proof}

\begin{proposition}\label{prop4.5}
 Let $L$ be a nilpotent Lie superalgebra with $\dim L= (m\mid n), \dim L'=(r\mid s)$. Then there is no non-abelian nilpotent Lie superalgebra if $r+s=1$, and $s(L)= 1$.
\end{proposition}

\begin{proof}
 Suppose $\dim (L)=(m\mid n)$ and $\dim (L')=(1 \mid 0)$. Using Proposition \ref{prop4.4},  for some $p, q$ with $p+q\geq 1$,
 \begin{equation*}
   L \cong H(p, q)+ A(m-2p-1 \mid n-q),
 \end{equation*}
and moreover, if $p=0, q=1$ or $p+q \geq 2$ then $s(L)=2$. If $\dim L'=(0 \mid 1)$ then for some $p \geq 1$
 \begin{equation*}
   L \cong H_p + A(m-p \mid n-p-1).
 \end{equation*} 
 But then $s(L)=3$ when $p \geq2$ and $s(L)=2$ when $p=1$. So we only need to check there is no nilpotent Lie superalgebra with $\dim L'=(1 \mid 0)$ and $s(L)=1$ when  $p=1, q=0$. Then  we have
$ L \cong H(1, 0)+ A(m-3\mid n)$ and thus
\begin{equation}\label{eq4.6}
\dim \mathcal{M}(L)=2+\frac{1}{2}[(m-3+n)^{2}+n-m+3]+2(m-3+n),
\end{equation}
hence
\begin{align*}
\frac{1}{2}[(m+n-2)(m+n-1)]+n+1-s(L)&=2+\frac{1}{2}[(m-3+n)^{2}+n-m+3]+2(m-3+n)
 \end{align*}
 which implies $s(L)=0$. 
\end{proof}
\begin{remark}\label{rem510}
With Theorem \ref{th3.9}, for a non-abelian nilpotent Lie superalgebra $L$, 
$s(L)=0$ if and only if 
$$
L \cong 
H(1, 0) \oplus A(m-3 \mid n). \quad 
$$
\end{remark}
The following result holds for nilpotent Lie superalgebras with nilpotency class at least two.

\begin{proposition}\label{prop4.6}
 Let $L$ be a finite dimensional nilpotent Lie superalgebra with $\dim L=(m\mid n)$ and $\dim L'=(r\mid s)$ such that $r+s=2$, and $\dim Z(L)=1$. Then,
 $$
 \dim (L^{2})+ \dim  \mathcal{M}(L) \leq \dim \mathcal{M}(L/L^{2})+ \dim (L/Z_{2}(L) \otimes L^{2}).
 $$
\end{proposition}
\begin{proof}
 Let $L \cong F/R$ where $F$ is a free Lie superalgebra and $R$ is a graded ideal. Since $L/L^{2} \cong \frac{(F/R)}{(F/R)^{2}} \cong \frac{F}{F^{2}+R}$, on one hand we have
 \begin{align*}
  \mathcal{M}(L/L^{2}) &= \frac{F' \cap (F^{2}+R)}{[F, F^{2}+R]}\\
 &\cong \frac{(R \cap F')+F^{2}/[R, F]}{[F, F^{2}+R]/[R, F]},
 \end{align*}
 and on the other hand,
 $$
 \frac{(R\cap F')+ F^{2}/[R, F]}{R \cap F'/[R, F]} \cong \frac{(R\cap F')+ F^{2}}{R \cap F'} \cong \frac{F^{2}}{F^{2}\cap R \cap F'}= \frac{F^{2}}{F^{2}\cap R} \cong \frac{F^{2}+R}{R}
 \cong L^{2}.
$$
We get,
$$
\dim L^{2}+ \dim \mathcal{M}(L)= \dim  \mathcal{M}(L/L^{2})+ \dim ([R+F^{2}, F]/[R, F]).
$$
Since $L$ is a nilpotent Lie superalgebra $L^{2} \neq L'$, and $\dim L'=2$, therefore $L^{2} = Z(L)$. Now put $Z_{k}(L)=S_{k}/R$ for $0\leq k \leq 1$, and consider the mapping
\begin{align*}
\theta: \frac{F^{2}+R}{R} \otimes \frac{F}{S_{2}} &\rightarrow \frac{[R+F^{2}, F]}{[R, F]}\\
 \theta(\hat{a}, \bar{x}) &=\overline{[a, x]}.
 \end{align*}
Here $\theta$ is a bilinear surjective map, $\frac{[R+F^{2}, F]}{[R, F]}$ is an epimorphic image of $L/Z_{2}(L) \otimes L^{2}$ as required.
\end{proof}

For non-abelian nilpotent Lie superalgebras $L$, the results below show that how the value of $s(L)$ put restrictions on dimension of $L$.
\begin{proposition}\label{prop4.7}
  Let $L$ be a non-abelian nilpotent Lie superalgebra of dimension $(m\mid n)$ with $\dim L'=(r \mid s), r+s=2$ and $\dim Z(L)=(a\mid b)$, $a+b=1$.
 \begin{enumerate}
\item There is no nilpotent Lie superalgebra $L$ with dimension $m+n \geq 5$ when $s(L)=1$.
\item There is no nilpotent Lie superalgebra $L$ with dimension $m+n \geq 6$ when $s(L)=2$.
\end{enumerate}  
\end{proposition}
\begin{proof}
 Suppose $L$ is nilpotent with $\dim L'=2$ and $\dim Z(L)=1$. Hence $\dim L^{2}=1$, and $L^{2}=Z(L)$. Clearly $Z(L) \subset L'$, and $(L/Z(L))^{'}=L'/Z(L)$. Consider $\dim Z(L)=(1\mid 0)$, then $\dim (L/Z(L))'=(1\mid 0)$ or $(0\mid 1)$. Assume $\dim (L/Z(L))'=(1\mid 0)$, $\dim L'=(2 \mid 0)$. By invoking Proposition \ref{prop4.4}, it follows that
 \begin{equation*}
  L/Z(L) \cong H(p, q)+ A(m-2p-2\mid n-q),
 \end{equation*} 
 for some non-negative integers $p, q$ with $p+q \geq 1$. When $p+q \geq 2$ we have
\begin{equation*}
 \dim \mathcal{M}(L/Z(L))=\frac{1}{2}[(m-2+n)^{2}+n-m)].
\end{equation*}
Clearly $\dim (L' \cap Z(L))= 1$, so using Theorem \ref{th3.2},
\begin{equation}\label{eq4.7}
 \dim \mathcal{M}(L)+1 \leq \frac{1}{2}[(m-2+n)^{2}+n-m]+m-2+n.
\end{equation}
Thus, 
\begin {align*}
 \frac{1}{2}[(m+n)^{2}-3m-n]+3-s(L) &\leq \frac{1}{2}[(m+n)^{2}-3m-n]
\end {align*}
which implies  $3  \leq s(L)$.
Similarly for $p=0, q=1$, we find $s(L) \geq 3$. Remains to deal with the case when $p=1, q=0$. We have $\frac{L}{Z(L)} \cong H(1, 0) \oplus A(m-4\mid n),$ hence $\dim Z_{2}(L)=(m-2\mid n)$, and
\begin{equation*}
  \dim  \mathcal{M}\left(\frac{L}{Z(L)}\right)=4+\frac{1}{2}[(m+n)^{2}-5m-3n] = \dim  \mathcal{M}\left(\frac{L}{L^{2}}\right).
 \end{equation*}
 By Proposition \ref{prop4.6}
 \begin{equation}\label{eq4.8}
 1+ \dim \mathcal{M}(L) \leq 4+ \frac{1}{2}[(m+n)^{2}-5m-3n]+2.
 \end{equation}
 Thus, 
\begin{align*}
 1+ \frac{1}{2}[(m+n)^{2}-3(m+n)+2]+n+1-s(L) & \leq 6+\frac{1}{2}[(m+n)^{2}-5m-3n]
 \end{align*}
 gives  $m+n-s(L) \leq 3$.
If $s(L)=1$ or $s(L)=2$ there is no non-abelian nilpotent Lie superalgebra when $\dim L \geq 5$ or $\dim L \geq 6$ respectively.
\par
Now assume $\dim (L/Z(L))'=(0\mid 1)$, i.e., here $\dim L'=(1 \mid 1)$. By invoking Proposition \ref{prop4.4}, it follows that with $p \geq 1$,
$$ L/Z(L) \cong H_p \oplus A(m-1-p \mid n-p-1).$$
Proceeding as before with $p \geq 2$ we get $s(L) \geq 4$. With $p=1$ we have $L/Z(L) \cong H_1 \oplus A(m-2 \mid n-2)$. Here $\dim Z_{2}(L)=(m-2 \mid n-2)$ and using Proposition \ref{prop4.6}, we get $m+n \leq 4$ when $s(L)=1$ and $m+n \leq 5$ when $s(L)=2$.
\par
Now consider $\dim Z(L)=(0\mid1)$. Here again $\dim (L/Z(L))^{'}=(1 \mid 0)$ or $(0 \mid 1)$. Assume $\dim (L/Z(L))^{'}=(1 \mid 0)$, so $\dim L'=(1 \mid 1)$. Then for $p+q \geq 1$
  \begin{equation*}
  L/Z(L) \cong H(p, q)+ A(m-2p-1\mid n-q-1).
 \end{equation*}
 When $p+q \geq 2$ or $p=0, q=1$ we arrive at  $s(L) \geq 3$. For $p=1, q=0$ as $\frac{L}{Z(L)} \cong H(1, 0)\oplus A(m-3\mid n-1)$, so $\dim Z_{2}(L)=(m-2\mid n)$, and 
\begin{equation*}
 \dim  \left(\mathcal{M}\Big(\frac{L}{L^{2}}\Big)\right)=-1+\frac{1}{2}[(m+n-2)^{2}+n-m].
 \end{equation*}
 Now invoking Proposition \ref{prop4.6} with $s(L)=1$ we get $m+n \leq 4$ and with $s(L)=2$ we get $m+n \leq 5$. Finally by letting $(L/Z(L))^{'}=(0 \mid 1)$ and with $p \geq 1$, $L/Z(L) \cong H_{p} \oplus A(m-p \mid n-p-2)$. When $p \geq 2$ we get $s(L) \geq 4$. For $p=1$, we get $m+n \leq 2$ with $s(L) =1$, and $m+n \leq 3$ with $s(L)=2$.
\end{proof}
\begin{proposition}\label{prop4.8}
 Let $L$ be a nilpotent Lie superalgebra of $\dim L=(m\mid n)$. Then there is no nilpotent Lie superalgebra for $L'=Z(L)$, $\dim Z(L)=2$, when $m+n \geq 6$ with $s(L)\leq 2$.
\end{proposition}
\begin{proof}
Suppose that $L$ is nilpotent Lie superalgebra with $m+n \geq 6$ along with $L'=Z(L)$, and $\dim Z(L)=2$. Let $I$ be a central ideal with $\dim I=(a \mid b), a+b=1$. First assume that $\dim I=(1 \mid 0)$. Further suppose $\dim L'=(2 \mid 0)= \dim Z(L)$, then $\dim (L/I)^{'}=(1 \mid 0)$. Using Proposition \ref{prop4.4} for some non-negative integers $p, q$, and  $p+q \geq 1$ we have
 $$L/I \cong H(p, q)+ A(m-2p-2\mid n-q).$$  
Now, suppose there are two distinct central ideals $I_{1}$ and $I_{2}$ of $L$.  Clearly $\dim I_{i}=(1\mid 0)$, for $i=1,2$. For any ideal $I$ of $L$, $L'/I \subseteq Z(L/I)$, so we can write $S_{i}/I_{i} \oplus L'/I_{i}=Z(L/I_{I})$, and hence $L/I_{i}=T_{i}/I_{i}\oplus S_{i}/I_{i}$, for $ i=1,2$. Then with $p=1, q= 0$, we have $T_{i}/I_{i}\cong H(1,0)$, and $S_{i}/I_{i} \cong A(m-4\mid n-1)$. Again as $L'=Z(L)$, and  $L'/I_{i} \cap S_{i}/I_{i}=\{0\}$, so $L' \cap S_{i}\subseteq I_{i}$. We get $S_{1} \cap S_{2} \cap Z(L) \subseteq I_{1} \cap I_{2}=\{0\}$ which implies that $S_{1} \cap S_{2}=\{0\}$. Again $\dim S_{1} =(m-3 \mid n)=\dim S_{2}$, and as $\dim L=m+n \geq 6$ thus, $L=S_{1}\oplus S_{2}$, and $m=6, n=0$. Clearly $\dim (S_{i}^{'} \cap I_{i})=1$.  Using Theorem \ref{th3.2},
\begin{align*}
\dim \mathcal{M}(S_{i})+1 &\leq \frac{1}{2}[(m-4+n)^{2}+n-m+4]+m+n-4\\
\mathcal{M}(S_{i}) & \leq 2.
\end{align*}
Now \begin{align*}
     \dim \mathcal{M}(L) &\leq 2+2+(m+n-4)^{2}
    \end{align*}
    implies $3 \leq s(L)$. Hence with $s(L) \leq 2$ there is no nilpotent Lie superalgebra when $m+n \geq 6$.

Consider $L$ has only one central graded ideal $I$.  Thus for $p+q \geq 2$ and using Corollary \ref{cor3.6}, $
 \dim (L/I) \leq \frac{1}{2}[(m+n-2)^{2}+n-m]$.
 Now by Theorem \ref{th3.2} 
 \begin{equation*}
  \dim \mathcal{M}(L)+ \dim (L'\cap I) \leq \dim (\mathcal{M}(L/I))+\dim \mathcal{M}(I)+\dim \left(\frac{L}{L'},\otimes I\right),
 \end{equation*}
 and substituting for $\dim \mathcal{M}(L)$
\begin{align*}
\frac{1}{2}[(m+n-1)(m+n-2)]+n+1-s(L)+1 &\leq \frac{1}{2}[(m+n-2)^{2}+n-m]+m+n-2.
\end{align*}    
Simplifying the above inequality we get  $3 \leq s(L)$. Similarly when $p=0, q=1$ we get $s(L) \geq 3$.  
  
  \smallskip
  
Suppose $\dim L'= (1 \mid 1)$, then $\dim (L/I)^{'}=(0 \mid 1)$. Using Proposition \ref{prop4.4}, with $p \geq 1$, $$L/I \cong H_{p} \oplus A(m-p-1 \mid n-p-1).$$    
 When $p \geq 2$, $\dim \mathcal{M}(L) \leq -1+ \frac{1}{2}[(m+n-2)^{2}+n-m]$, and hence using Proposition \ref{prop4.6}, $s(L) \geq 4$. Further when $p=1$, $s(L) \geq 3$.
    \par

Next consider $\dim I=(0\mid 1)$ and assume $\dim L'=(0 \mid 2)$. So $\dim(L/I)^{'}=(0 \mid 1)$ and hence with $p \geq 1$, $$L/I \cong H_{p} \oplus A(m-p\mid n-p-2).$$    
Consider $p \geq 2$, then $\dim \mathcal{M}(L/I)= -2 +\frac{1}{2}[(m+n-2)^{2}+n-m]$, and finally applying Theorem \ref{th3.2} we get $s(L) \geq 4$. Similarly with $p=1$ we get $s(L) \geq 3$.

\par

At last, suppose $\dim I=(0\mid 1)$, and assume $\dim L'=(1 \mid 1)$, then $\dim(L/I)^{'}=(1 \mid 0)$. By using Proposition \ref{prop4.4}, we have 
 $L/I \cong H(p, q)+ A(m-2p-1\mid n-q-1)$ for some non-negative integers $p, q$, and $p+q \geq 1$. Here with $p+q \geq2$, $\dim \mathcal{M}(L/I)= -1 + \frac{1}{2}[(m+n-2)^{2}+n-m]$. By applying Theorem $\ref{th3.2}$ we get $\dim \mathcal{M}(L) +1 \leq \dim (\mathcal{M}(L/I))+1+m+n-2$ which implies $s(L) \geq 3$. Similarly we compute $s(L) \geq 3$ when $p=0, q=1$. Consider $p=1, q=0$, so $L/I \cong H(1, 0) \oplus A(m-3 \mid n-1)$, and $\dim (\mathcal{M}(L/I))=7+\frac{1}{2}[(m+n)^{2}-7m-5n]$. Finally using Theorem \ref{th3.2}, $m+n-s(L) \leq 3$. With $s(L)=1$ (resp. $s(L)=2$) we get $m+n \leq 4$ (resp. $m+n \leq 5$).
\end{proof}

\section{Proof of main results}
\subsection{Proof of Theorem \ref{th4.9a}} 
Suppose $L$ is a non-abelian nilpotent Lie superalgebra and $\dim L= (m\mid n )$, then using Remark \ref{rem510}, \[L\cong H(1, 0)\oplus A(m-3\mid n).\]  If $L$ is abelian then using Corollary \ref{cor l =2} and Theorem \ref{th3.1},
\[L\cong A(2 \mid 0), \; A(0 \mid 2), \; A(1 \mid 1).\]

\subsection{Proof of Theorem \ref{th4.9}}   
If $L$ is abelian then with Lemma \ref{lem4.1}, \[L \cong A(1 \mid 0), A(0 \mid 1).\]  Consider $L$ is non-abelian nilpotent then from Propositions \ref{prop4.2}, \ref{prop4.3}, \ref{prop4.4a}, \ref{prop4.7} and \ref{prop4.8}, we have 
\[\dim L' = 2, L' =Z(L) \quad \mbox{and} \quad m+n = 5.\]
Thus, using Theorem \ref{th4.6}, $L$ is isomorphic to one of the Lie superalgebras in Table \ref{tab4}. But with $s(L)=1$, $L \cong L_{5, 0}^2$. \qed

\subsection{Proof of Theorem \ref{th4.10}}
By Proposition \ref{prop4.2}, $\dim L' \leq 2$. There is no abelian nilpotent Lie superalgebra with $s(L)=2$. Suppose $\dim L'= (r \mid s)$, $r+s=1$, and  $s(L)=2$. If $\dim L' = (1\mid 0)$, by Proposition \ref{prop4.4},  
\[ L \cong H(0, 1)\oplus A(m-1 \mid n-1)\] or, 
 \[ L \cong H(p, q)\oplus A(m-2p-1 \mid n-q)\] where $p+q \geq 2$.
 When $\dim L'=(0\mid 1)$ then \[ L \cong H_1\oplus A(m - 1 \mid n-2).\]
   Now consider $\dim L'=(r \mid s), r+s=2$. Invoking Proposition \ref{prop4.4a} we have $\dim Z(L) \leq 3$, and $L' \subseteq Z(L)$. If $\dim Z(L)=3$ then $\dim L=(m \mid n), m + n \geq 4$. Assume $\dim L=4$ with $\dim Z(L)=3$, then this contradicts the fact that $\dim L'=2$. If $\dim Z(L)=2$, by Theorem \ref{th4.2}, there is no such $L$ along with $s(L)=2$. If $\dim Z(L)=1$, and $s(L) =2$ then by Theorem \ref{th4.3}, $L \cong L_{4,0}^1$. Now assume $m+n \geq 5$. If $\dim Z(L)=1$ by Proposition \ref{prop4.7} we have $m +n \leq 5$, and hence $\dim L=5$. Using Theorem \ref{th4.5} there is no such Lie superalgebra with $s(L)=2$. If $\dim Z(L)=2$, then $ L'= Z(L)$, and by Proposition \ref{prop4.8} we have $m +n \leq 5$ thus $\dim L=5$. By Theorem \ref{th4.6},  no such Lie superalgebra exists.  If $\dim Z(L)=3$, and $ L' \subseteq Z(L)$, choose an one dimensional central graded ideal $I$ of $L$ such that $L' \cap I = \{0\}$. First take $\dim I=(1 \mid 0)$. By Theorem \ref{th3.9}
 $$\dim \mathcal{M}(L/I) \leq \frac{1}{2}(m+n-1)(m+n-4)+n+1,$$ and using Theorem \ref{th3.2}, $$\frac{1}{2}(m+n-2)(m+n-1)+n-1= \dim \mathcal{M}(L)\leq \dim \mathcal{M}(L/I)+m+n-3.$$ Thus $\dim \mathcal{M}(L/I)=\frac{1}{2}(m+n-1)(m+n-4)+n+1$. So $s(L/I)=1$, and by Theorem \ref{th4.9}, $L/I \cong L_{5,0}^2$, 
\[L \cong L_{5, 0}^2 \oplus A(1 \mid 0).\]
If $\dim I=(0 \mid 1)$ then by proceeding  as above we get $s(L/I)=1$, and hence $L/I \cong L_{5, 0}^2$ which implies \[L \cong L_{5,0}^2 \oplus A(0 \mid 1).\] 
This completes the proof of the theorem. \qed

\subsection{Proof of Theorem \ref{th410}}

Consider the non-abelian nilpotent Lie superalgebra $L$ with $\dim L=(m \mid n)$.
If $t(L)=1$ then, $\dim L=\frac{1}{2}[(m+n)^{2}+n-m]-1$. Using \eqref{eq3} we further have $s(L)=3-(m+n)$ which implies $m+n= 2$ or $3$. If $\dim L=2$ then $s(L)=1$, and there is no such superalgebra. If $\dim L=3$ then $s(L)=0$, and clearly $L \cong H(1, 0)$. Conversely if $L \cong H(1,0)$ then $s(L)=\frac{1}{2}(2)+1- \dim \mathcal{M}(L)=0$ which clearly implies $3-t(L)=2-s(L)$,i.e. $t(L)=1$.

\smallskip

 If $t(L)=2$ then $\dim \mathcal{M}(L)=\frac{1}{2}[(m+n)^{2}+n-m]-2$, and then $s(L)=4-(m+n)$, and hence $2\leq \dim L \leq 4$. Let $m+n=2$,  more precisely  $m=1, n=1$ then $s(L)=2$ and $L \cong H(0,1)$. Conversely let $L \cong H(0,1)$, and we know $\dim \mathcal{M}(H(0,1))=0$, so clearly $t(L)=2$. Consider $\dim L=3$ then $s(L)=1$, and there is no such Lie superalgebra. If $m+n=4$ then $s(L)=0$, and $L\cong H(1,0) \oplus A(1 \mid 0)$. Conversely if $L\cong H(1,0) \oplus A(1 \mid 0)$, then $\dim L=(4 \mid 0)$, and $\dim \mathcal{M}(L)=4$. Altogether this gives $t(L)=2$.
 
 \smallskip
 
 If $t(L)=3$, then $s(L)=5-(m+n)$, and hence $3 \leq \dim L \leq 5$. Consider $m+n=3$ then $s(L)=2$. If $\dim L=(3\mid 0)$ there is no superalgebra with $s(L)=2$. Take $\dim L=(2 \mid 1)$ then with $s(L)=2$ we get $L \cong H(0,1)\oplus A(1 \mid 0)$. Conversely $t(L)=4-\dim \mathcal{M}(L)$ implies $t(L)=3$. Consider $\dim L=(1 \mid 2)$ then with $s(L)=2$, $L \cong H(0,2),$ or $H(0,1)\oplus A(0 \mid 1)$ or $H_{1}$. Conversely if $L \cong H(0,2)$ then $\dim L =2$ and hence $t(L)=2$.  Similarly for the rest superalgebras  we get $t(L)=2$. Consider  $\dim L=4$ then $s(L)=1$, and there is no such superalgebra. Finally, if $\dim L=5$ then $s(L)=0$ if and only if $L \cong H(1, 0)\oplus A(2 \mid 0)$.
 
 \smallskip
 
 If $t(L)=4$ then $s(L)=6-(m+n)$ implies $3 \leq \dim L \leq 6$. Let $m+n=3$ then $s(L)=3$. Certainly $ \dim L \neq (0 \mid 3)$, and if $ \dim L= (3 \mid 0)$ then $s(L)=0$. Consider $\dim L =(2\mid 1)$, and let $\{x_{0}, y_{0}, x_{1}\}$ be a basis for $L$ where $x_{0}, y_{0} \in L_{\bar{0}}$, $x_{1} \in L_{\bar{1}}$. Only possible non-vanishing bracket is $[x_{1}, x_{1}]=x_{0}$, and so $\dim L'=(1 \mid 0)$. Hence $L \cong H(0, 1)\oplus A(1 \mid 0)$ which implies $s(L)=2$. Consider $\dim L=(1 \mid 2)$, and $\{x_{0}, x_{1}, y_{1}\}$ be a basis for $L$ where $x_{0} \in L_{\bar{0}}$, $x_{1}, y_{1} \in L_{\bar{1}}$. If $\dim L'=( 0 \mid 1)$ then  $[x_{0} x_{1}]=y_{1}$, $L \cong H_{1}$. If $\dim L'=(1 \mid 0)$ then $L \cong H (0, 2)$. So there is no nilpotent Lie superalgebra when $\dim L=3$ with  $s(L)=3$. Consider $m+n=4$ then $s(L)=2$. By using Theorem \ref{th4.10}, $L \cong$  
 $L_{4,0}^1$, or $H(0,3)$, or $H_{1} \oplus A(1 \mid 0)$, or $H(0,1)\oplus A(1 \mid 1)$,or $H(0,1)\oplus A(0 \mid 2)$, or $H(0, 1)\oplus A(2 \mid 0)$ or $H_{1}\oplus A(0 \mid 1)$  or $H_{1} \oplus A(2 \mid 0)$. If $\dim L=5$ then $s(L)=1$ if and only if $L \cong L_{5,0}^{2}$. Finally if $\dim L =6$ then $s(L)=0$ if and only if $L \cong H(1,0) \oplus A(3 \mid 0)$. \qed

\section{Acknowledgement}

We would like to thank the referee for  his/her comments which helped us to improve the readability of this paper. The author's research is supported by NBHM Post-Doctoral Fellowship, Govt. of India.

\medskip

\end{document}